\documentclass[10pt,draft,reqno]{amsart}
     \makeatletter
     \def\section{\@startsection{section}{1}%
     \z@{.7\linespacing\@plus\linespacing}{.5\linespacing}%
     {\bfseries
     \centering
     }}
     \def\@secnumfont{\bfseries}
     \makeatother
\newtheorem{theorem}{Theorem}[section]
\newtheorem{lemma}[theorem]{Lemma}
\newtheorem{proposition}[theorem]{Proposition}
\newtheorem{corollary}[theorem]{Corollary}
\theoremstyle{definition}

\theoremstyle{remark}

\numberwithin{equation}{section}
\newcommand{\rF}{\mathcal{F}}

\newcommand{\rN}{\mathcal{N}}

\newcommand{\rP}{\mathcal{P}}
\newcommand{\rE}{\mathcal{E}}
\newcommand{\rL}{\mathcal{L}}
\newcommand{\rS}{\mathcal{S}}

\newcommand{\CC}{\mathbb{C}}

\newcommand{\EE}{\mathbb{E}}

\newcommand{\NN}{\mathbb{N}}

\newcommand{\PP}{\mathbb{P}}
\newcommand{\QQ}{\mathbb{Q}}
\newcommand{\RR}{\mathbb{R}}

\newcommand{\ZZ}{\mathbb{Z}}

\def\Rp{\RR^+}

\def\O{\Omega}
\def\wt{\widetilde}
\def\indic{{\mathop{\rm 1\mkern-4mu l}}}
\def\qq{\qquad}

\def\wh{\widehat}
\def\wt{\widetilde}

\def\indic{{\mathop{\rm 1\mkern-4mu l}}}

\def\ecarte{\vphantom{\buildrel\bigtriangleup\over =}}

\def\Rp{{\RR^{+}\!}}

\def\[{{\mathord{[\![}}}
\def\]{{\mathord{]\!]}}}

\def\norme#1{\left\| #1\right\|}
\def\normca#1{{\left\| #1\right\|}^2}
\def\ab#1{\left\vert #1\right\vert}

\def\seq#1{\left({#1}_n\right)}

\def\x{\chi}
\def\X{\x}

\def\O{\Omega}

\def\a{\alpha}

\def\s{\sigma}

\def\n{\nu}

\def\l{\lambda}

\def\o{\omega}

\begin{document}

\title[Markov Chains and Dynamical Systems]{Markov Chains and Dynamical Systems:\\The Open System Point of View\,*}

\author{St\'ephane ATTAL}
\thanks{* Work supported by ANR project ``HAM-MARK" N${}^\circ$ ANR-09-BLAN-0098-01}
\address{Universit\'e de Lyon\\ Universit\'e de Lyon 1, C.N.R.S.\\ Institut Camille Jordan\\ 21 av Claude Bernard\\ 69622 Villeubanne cedex, France}
\email{attal@math.univ-lyon1.fr}
\urladdr{http://math.univ-lyon1.fr/~attal}


\begin{abstract}
This article presents several results establishing connections between Markov chains and dynamical systems, from the point of view of open systems in physics. We show how all Markov chains can be understood as the information on one component that we get from a dynamical system on a product system, when losing information on the other component. We show that passing from the deterministic dynamics to the random one is characterized by the loss of algebra morphism property; it is also characterized by the loss of reversibility. In the continuous time framework, we show that the solutions of stochastic differential equations are actually deterministic dynamical systems on a particular product space. When losing the information on one component, we recover the usual associated Markov semigroup. 
\end{abstract}

\maketitle


\section{Introduction}
This article aims at exploring the theory of Markov chains and Markov processes from a particular point of view. This point of view is very physical and commonly used in the theory of {\it open systems}. Open systems are physical systems, in classical or in quantum mechanics, which are not closed, that is, which are interacting with another system. In general the system we are interested in is ``small" (for example, it has only a finite number of degrees of freedom), whereas the outside system is very large (often called the``environment", it may be a heat bath typically). 

This is now a very active branch of research to study such systems coupled to an environment. In classical mechanics they are used to study conduction problems (Fourier's law for example, see \cite{O-B}, \cite{Rey}) but more generally out of equilibrium dynamics (see \cite{Rue}, \cite{Bod}). In quantum mechanics, open systems appear fundamentally for the study of decoherence phenomena (see \cite{Spe}), but also it is the basis of quantum communication (see \cite{Pre}). Problems of dissipation, heat conduction, out of equilibrium dynamics in quantum mechanics (see \cite{JOP}, \cite{K-P}) lead to very important problems which are mostly not understood at the time we write this article.

The aim of this article is to make clear several ideas and connections between deterministic dynamics of closed systems, effective dynamics of open systems and Markov processes. 

\smallskip
Surprisingly enough these ideas are made rather clear in the literature when dealing with the quantum systems, but not that much with classical ones! Indeed, it is common in quantum mechanics to consider a bipartite system on which one component is not accessible (it might be an environment which is too complicated to be described, or it might be Bob who is sharing the photons of a correlated pair with Alice, in Quantum Information Theory). It is well-known that, tracing out over one component of the system, the unitary Schr\" odinger dynamics becomes described by completely positive maps, in discrete time, or completely positive semigroups, in continuous time. 

In \cite{A-P}, for example, the authors show how every discrete time semigroup of completely positive maps can be described by a realistic physical system, called ``Repeated Quantum Interactions". They show that in the continuous time limit these Hamiltonian dynamics spontaneously converge to a dynamics described by a quantum Langevin equation. 

\medskip
In this article we establish many similar results in the context of classical dynamical systems and Markov chains. The article is structured as follows. 
In Section \ref{S:markov_dynamical}, we show that Markov chains appear from any dynamical system on a product space, when averaging out one of the two components. This way, Markov chains are interpreted as what remains on one system when it interacts with some environment but we do not have access to that environment. The randomness appears directly for the determinism, solely by the fact that we have lost some information. We show that any Markov chain can be obtained this way. We also show two results which characterize what properties are lost when going from a deterministic dynamical system to a Markov chain: typically the loss of algebra morphism property and the loss of reversibility. 

In Section \ref{S:SDE_dynamical} we explore the context of classical Markov process in the continuous time setup. We actually concentrate on stochastic differential equations. Despite of their ``random character", we show that stochastic differential equations are actually deterministic dynamical systems. They correspond to a natural dynamical system which is used to dilate some Markov processes into a deterministic dynamics. The role of the environment is played by the canonical probability space (here the Wiener space), the action of the environment is the noise term in the stochastic differential equation.

\section{Markov Chains and Dynamical Systems}\label{S:markov_dynamical}

\subsection{Basic Definitions}

Let us recall some basic definitions concerning dynamical systems and Markov chains.

\smallskip
Let $(E,\rE)$  be a measurable space. Let  $\wt T$ be a
measurable function from $E$ to $E$.
We then say that $\wt T$ is a \emph{dynamical
system} on $E$. 
Such a mapping $\wt T$ induces a natural mapping $T$ on
 ${\rL}^\infty(E)$ defined by
$$
Tf(x)=f(\wt T x)\,.
$$
Note that this mapping clearly satisfies the following properties (proof left to the reader).

\begin{proposition}\label{P:dynamical}\ 

\smallskip
\noindent i) $T$ is a $*$-homomorphism of the $*$-algebra $\rL^\infty(E)$,

\smallskip
\noindent ii) $T(\indic_E)=\indic_E$,

\smallskip
\noindent iii) $\norme{T}= 1$.
\end{proposition}

\smallskip
What is called \emph{dynamical system} is actually the associated discrete-time semigroup $({\wt T}^n)_{n\in\NN}$, when acting on points, or $(T^n)_{n\in\NN}$, when acting on functions. 

When the mapping $\wt T$ is invertible, then so is the associated operator $T$. The semigroups $({\wt T}^n)_{n\in\NN}$ and $(T^n)_{n\in\NN}$ can then be easily extended into one-parameter groups $({\wt T}^n)_{n\in\ZZ}$ and $(T^n)_{n\in\ZZ}$, respectively. 

\bigskip
Let us now recall basic definitions concerning Markov chains.
Let $(E,\rE)$ be a measurable space. A mapping $\nu$ from $E\times \rE$ to $[0,1]$ is a \emph{Markov kernel} if 

\smallskip
\noindent i) $x\mapsto \nu(x,A)$ is a measurable function, for all $A\in\rE$,

\smallskip
\noindent ii) $A\mapsto \nu(x,A)$ is a probability measure, for all $x\in E$. 

\smallskip
When $E$ is a finite set, then $\nu$ is determined by the quantities 
$$
P(i,j)=\nu(i,\{j\})
$$
which form a \emph{stochastic matrix}, i.e. a square matrix with positive entries and sum of each row being equal to 1.

\smallskip
In any case, such a Markov kernel $\nu$ acts on $\rL^\infty(E)$ as follows:
$$
\n\circ f(x)=\int_E f(y)\, \n(x,dy)\,.
$$

\smallskip
A linear operator $T$ on $\rL^\infty(E,\rE)$ which is of the form
$$
Tf(x)=\int_E f(y)\, \n(x,dy)\,,
$$
for some Markov kernel $\n$, is called a \emph{Markov operator}. 

\smallskip
In a dual way, a Markov kernel $\n$ acts on probability measures on $(E,\rE)$. Indeed, if $\PP$ is a probability measure on $(E,\rE)$ then so is the measure $\PP\circ\n$ defined by 
$$ 
\PP\circ\nu(A)=\int_E \nu(x,A)\,\PP(dx)\,.
$$

\smallskip
Finally, Markov kernels can be composed. If $\nu_1$ and $\nu_2$ are two Markov kernels on $(E,\rE)$ then so is 
$$
\n_1\circ\n_2(x,A)=\int_E \n_2(y,A)\, \n_1(x,dy)\,.
$$
This kernel represents the Markov kernel resulting from making a first step following $\n_1$ and then another step following $\n_2$. 

\medskip
A \emph{Markov chain} with \emph{state space} $(E,\rE)$ is a discrete-time stochastic process $(X_n)_{n\in\NN}$ defined on a probability space $(\Omega,\rF,\PP)$ such that each $X_n\,:\,\O\rightarrow E$ is measurable and
$$
\EE\left[f(X_{n+1})\,\vert\, X_0,X_1,\ldots,X_n\right]=\EE\left[f(X_{n+1})\,\vert\,X_n\right]
$$
for all bounded function $f\,:\, E\rightarrow \RR$ and all $n\in \NN$. In particular, if $\rF_n$ denotes the $\s$-algebra generated by $X_0,X_1,\ldots, X_n$, then the above implies
$$
\EE\left[f(X_{n+1})\,\vert\,\rF_n\right]=L_nf(X_{n})
$$
for some function $L_nf$. The Markov chain is \emph{homogeneous} if furthermore $L_n$ does not depend on $n$. We shall be interested only in this case and we denote by $L$ this unique value of $L_n$:
\begin{equation}\label{markov}
\EE\left[f(X_{n+1})\,\vert\,\rF_n\right]=Lf(X_{n})\,.
\end{equation}
Applying successive conditional expectations, one gets
$$
\EE\left[f(X_{n})\,\vert\,\rF_0\right]=L^nf(X_{0})\,.
$$
 
If $\nu(x,dy)$ denotes the conditional law of $X_{n+1}$ knowing $X_n=x$, which coincides with the  conditional law of $X_{1}$ knowing $X_0=x$, then $\nu$ is a Markov kernel and one can easily see that
$$
Lf(x)=\int_E f(y)\, \nu(x,dy)=\nu\circ f(x)\,.
$$
Hence $L$ is the Markov operator associated to $\nu$.

With our probabilistic interpretation we get easily that $\n\circ f(x)$ is the expectation of $f(X_1)$ when $X_0=x$ almost surely. 
The measure $\PP\circ\n$ is the distribution of $X_1$ if the distribution of $X_0$ is $\PP$.

\smallskip
We end up this section with the following last definition. 
 A Markov kernel $\n$ is said to be \emph{deterministic} if for all $x\in E$ the measure $\n(x,\,\cdot\,)$ is a Dirac mass. This is to say that there exists a measurable mapping $\wt T\,:\,E\rightarrow E$ such that 
$$
\n(x,dy)=\delta_{\wt T(x)}(dy)\,.
$$
In other words, the Markov chain associated to $\n$ is not random at all, it maps with probability 1, each point $x$ to  $\wt T(x)$: it is a dynamical system.

\subsection{Reduction of Dynamical Systems}

Now consider two measurable spaces $(E,\rE)$ and $(F,\rF)$, together
with a dynamical system $\wt T$ on $E\times F$, equipped with the product $\s$-field. As above, consider the
lifted mapping $T$ acting on ${\rL}^\infty(E\times F)$. 

For any
bounded measurable function $f$ on $E$, we consider the bounded (measurable) function
$f\otimes \indic$ on $E\times F$ defined by
$$
(f\otimes\indic)(x,y)=f(x)\,,
$$
for all $x\in E$, $y\in F$. 

Assume that $(F,\rF)$ is equipped with a probability measure $\mu$. We shall be interested in the mapping $L$ of ${\rL}^\infty(E)$
defined by
\begin{equation}\label{E:markov_dynam}
Lf(x)=\int_F T(f\otimes \indic)(x,y)\,d\mu(y)=\int_F (f\otimes \indic)\wt T(x,y)\,d\mu(y)\,.
\end{equation}
In other words, we have a deterministic dynamical system on a product
space. We place ourselves from one component point of view only (we  have access
to $E$ only). Starting from a point $x\in E$ and a function $f$ on $E$
we want to see how they evolve according to $T$, but seen from the $E$
point of view. The function $f$ on $E$ is naturally lifted into a
function $f\otimes \indic$ on $E\times F$, that is, it still acts on
$E$ only, but it is now part of a ``larger world''. We make $f\otimes \indic$ evolve
according to the deterministic dynamical system $T$. Finally, in order
to come back to $E$ we project the result onto $E$, by taking the
average on $F$ according to a fixed measure $\mu$ on $F$. This is to say that, from the set $E$, what we see of the action of the ``environment" $F$ is just an average with respect to some measure $\mu$.  

\begin{theorem}\label{T:dynamical_markov}
The mapping $L$ is a Markov operator on $E$.
\end{theorem}
\begin{proof}
As $\wt T$ is a mapping from $E\times F$ to $E\times F$, 
there exist two measurable mappings:
$$
X\ :\ E\times F\longrightarrow E\qq\hbox{and}\qq Y\ :\ E\times
F\longrightarrow F\,, 
$$
such that 
$$
\wt T(x,y)=(X(x,y),Y(x,y))
$$
for all $(x,y)\in E\times F$.  

Let us compute the quantity
$Lf(x)$, with these notations. We have
\begin{align*}
Lf(x)&=\int_F T(f\otimes\indic)(x,y)\,d\mu(y)\\
&=\int_F (f\otimes\indic)(X(x,y),Y(x,y))\, d\mu(y)\\
&=\int_F f(X(x,y))\, d\mu(y)\,.
\end{align*}
Denote by $\n(x, dz)$ the probability measure on $E$, which is the image of $\mu$ by the function
$X(x, \,\cdot\,)$ (which goes from $F$ to $E$, for each fixed $x$). By a standard result from Measure Theory, the Transfer Theorem, we  get 
$$
Lf(x)=\int_{E}f(z)\, \n(x,dz)\,.
$$
Hence $L$ acts on $\rL^\infty(E)$ as the  Markov transition kernel 
$\n(x,dz)$.
\end{proof}

\medskip
Note the following important fact: \emph{the mapping $Y$ played no role at all in the proof
above}. 

\smallskip
Note that the Markov kernel $\n$ associated to $\wt T$ restricted to $E$ is given by
\begin{equation}\label{E:nuxA}
\n(x,A)=\mu\left(\{y\in F;\, X(x,y)\in A\}\right)\,.
\end{equation}
In particular, when $E$ is finite (or even countable), the transition kernel $\n$ is
associated to a Markovian matrix $P$ whose
coefficients are given by
$$
P(i,j)=\n(i,\{j\})=\mu(\{k; X(i,k)=j\})\,.
$$

\smallskip
What we have obtained here is important and deserves more
explanation. Mathematically, we have obtained a commuting diagram:
$$
\begin{matrix}
&T&\\
\rL^\infty(E\times F)&\longrightarrow&\rL^\infty(E\times F)\\\\
\otimes\indic\Big\uparrow&&\Big\downarrow\mu\\\\
\rL^\infty(E)&\longrightarrow&\rL^\infty(E)\,.\\
&L&
\end{matrix}
$$
In more physical language, what we have obtained here can be interpreted in two different ways. If we think of the dynamical system $\wt T$ first, we have emphasized the fact that losing the information of a deterministic dynamics on one of the components creates a random behavior on the other component. The randomness here appears only as a lack of knowledge of deterministic behavior on a larger world. A part of the universe interacting with our system $E$ is inaccessible to us (or at least we see a very small part of it: an average) which results in random behavior on $E$.

In the converse direction, that is, seen from the Markov kernel point of view, what we have obtained is a \emph{dilation} of a
Markov transition kernel into a dynamical system. Consider the
kernel $L$ on the state space $E$. It does not represent the dynamics
of a closed system, it is not a dynamical system. In order to see $L$ as coming from a true dynamical system, we have
enlarged the state space $E$ with an additional state space $F$, which
represents the environment. The dynamical system $\wt T$ represents the
true dynamics of the closed system ``$E$+environment''. Equation
\eqref{E:markov_dynam} says exactly that the effective
pseudo-dynamics $L$ that we have observed on $E$ is simply due to the
fact that we are looking only at a subpart of a true dynamical system
and an average of the $F$ part of the dynamics.

\medskip
These observations would be even more interesting if one could prove the
converse: \emph{every Markov transition kernel can be obtained this
way}. This is what we prove now, with only a very small restriction on $E$.

\smallskip
Recall that a Lusin space is a measurable space which is homeomorphic (as a measurable space) to a Borel subset of a compact metrisable space. This condition is satisfied for example by all the spaces $\RR^n$. 

\begin{theorem}\label{T:dilate_kernel}
Let $(E,\rE)$ be a Lusin space and $\nu$ a Markov kernel on $E$. Then there exists a measurable space $(F,\rF)$, a probability measure $\mu$ on $(F,\rF)$ and a dynamical system $\wt T$ on $E\times F$ such that the Markov kernel $L$ associated to the restriction of $\wt T$ to $E$ is equal to $\nu$.
\end{theorem} 

\begin{proof}
Let $\nu(x, dz)$ be a Markov kernel on $(E,\rE)$.  Let $F$ be the set of 
functions from $E$ to $E$. For every finite subset
$\s=\{x_1,\ldots,x_n\}\subset E$ and every $A_1, \ldots, A_n\in\rE$ consider the set
$$
F(x_1,\ldots, x_n\,;\, A_1,\ldots,A_n)=\left\{y\in F\,;\ y(x_1)\in
A_1,\ldots,y(x_n)\in A_n\right\}\,.
$$
By the Kolmogorov Consistency Theorem (which applies for $E$ is is a Lusin space!) there exists a
unique probability measure $\mu$ on $F$ such that
$$
\mu\left(F(x_1,\ldots, x_n; A_1,\ldots,A_n)\right)=\prod_{i=1}^n \nu(x_i,
A_i)\,.
$$
Indeed, it is easy to check that the above formula defines a
consistent family of probability measures on the finitely-based
cylinders of $F$, then apply
Kolmogorov's Theorem.

\smallskip
Now define the dynamical system
$$
\begin{matrix}
\wt T\ :\ &E\times F&\longrightarrow&E\times F\\
&(x,y)&\longmapsto&(y(x),y)\,.
\end{matrix}
$$

With the same notations as in the proof of Theorem \ref{T:dynamical_markov},
we have $X(x,y)=y(x)$ in this particular case and hence
$$
\mu(\{y\in F\,;\ X(x,y)\in A\})=\mu(\{y\in F\,;\ y(x)\in A\})=\nu(x,A)\,.
$$
This proves our claim by (\ref{E:nuxA}).\end{proof}

\medskip
Note that in this dilation of $L$, the dynamical system $T$ has no reason to
be invertible in general. It is worth noticing that one can construct a dilation where $T$ is invertible. 

\begin{proposition}
Every Markov kernel $\n$, on a Lusin space $E$, admits a dilation $\wt T$ which is an invertible dynamical system.
\end{proposition}
\begin{proof}
Consider the construction and notations of Theorem \ref{T:dilate_kernel}. Consider the space $F'=E\times F$. Let $x_0$ be a fixed element of $E$ and define the mapping $\wt T'$ on $E\times F'$ by 
$$
\begin{cases}
\wt T'(x,(x_0,y))=(y(x), (x, y)),&\\
\wt T'(x,(y(x),y))=(x_0, (x, y)),&\\
\wt T'(x,(z,y))=(z, (x, y)),& \mbox{\ if } z\not = x_0\mbox{\ and }z\not = y(x)\,.
\end{cases}
$$
It is easy to check that $\wt T'$ is a bijection of $E\times F'$. Now extend the measure $\mu$ on $F$ to the measure
$\delta_{x_0}\otimes \mu$ on $F'$. Then the dynamical system $\wt T'$ is
invertible and dilates the 
same Markov kernel as $\wt T$.
\end{proof}

\subsection{Iterating the Dynamical System}\label{SS:RI}

We have shown that every dynamical system on a product set gives
rise to a Markov kernel when restricted to one of the sets. We have
seen that every Markov kernel can be obtained this way.
But one has to notice that our construction allows the dynamical
system $\wt T$ to dilate  the Markov
kernel $L$ as a single mapping only. That is, iterations of the
dynamical system $T^n$ do not in general dilate the semigroup $L^n$
associated to the Markov process. Let us check this with a simple
counter-example.

Put $E=F=\{1,2\}$. On $F$ define the probability measure $\mu(1)=1/4$
and $\mu(2)=3/4$. Define the dynamical system $\wt T$ on $E\times F$
which is the ``anticlockwise rotation'':
$$
\wt T(1,1)=(2,1),\qq \wt T(2,1)=(2,2),\qq \wt T(2,2)=(1,2),\qq \wt T(1,2)=(1,1)\,.
$$
With the same notations as in previous section, we have
$$
X(1,1)=2,\qq X(2,1)=2,\qq X(2,2)=1,\qq X(1,2)=1\,.
$$
Hence, we get
\begin{align*}
\mu(X(1,\,\cdot\,)=1)&=\frac 34\,,\qq
\mu(X(1,\,\cdot\,)=2)=\frac 14\,,\\
\ecarte \mu(X(2,\,\cdot\,)=1)&=\frac 34\,,\qq
\mu(X(2,\,\cdot\,)=2)=\frac 14\,.
\end{align*}
Hence the Markovian matrix associated to the restriction of $\wt T$ to
$E$ is 
$$
L=\left(\begin{matrix} \frac 34\,&\,\frac
14\\\ecarte\frac34\,&\,\frac14\end{matrix}\right)\,.
$$
In particular
$$
L^2=L\,.
$$
Let us compute ${\wt T}^2$. We get 
$$
{\wt T}^2(1,1)=(2,2),\qq {\wt T}^2(2,1)=(1,2),\qq {\wt T}^2(2,2)=(1,1),\qq {\wt T}^2(1,2)=(2,1)\,.
$$
Hence the associated $X$-mapping, which we shall denote by $X_2$, is given
by
$$
X_2(1,1)=2,\qq X_2(2,1)=1,\qq X_2(2,2)=1,\qq X_2(1,2)=2\,.
$$
This gives the Markovian matrix
$$
L_2=\left(\begin{matrix} 0\,&\,1\\\ecarte1\,&\,0\end{matrix}\right)\,,
$$
which is clearly not equal to $L^2$. 

\bigskip
It would be very interesting if one could find a dilation of the
Markov kernel $L$ by a dynamical system $\wt T$ such that any power
$T^n$ would also dilate $L^n$. We would have realized the whole
Markov chain as the restriction of iterations of a single dynamical
system on a larger space. 

This can be performed in the following way (note that this is not the
only way, nor the more economical). Let $L$ be a Markov operator on a Lusin space $E$ with kernel $\nu$ and let $T$ be a dynamical system on
$E\times F$ which dilates $L$. Consider the set
$\wh{F}=F^{\NN^*}$ equipped with the usual cylinder $\s$-field
$\rF^{\otimes\NN^*}$ and the product measure 
$\wh{\mu}=\mu^{\otimes\NN^*}$. The elements of $\wh F$ are sequences
${(y_n)}_{n\in\NN^*}$ in $F$. Put 
$$
\begin{matrix}
\wt{S}&:&E\times \wh F&\longrightarrow&E\times\wh F\\
&&(x,y)&\longmapsto&(X(x,y_1),\Theta(y))
\end{matrix}
$$
where $X$ is as in the the proof of Theorem \ref{T:dynamical_markov} and $\Theta$ is the usual \emph{shift} on $\wh F$:
$\Theta(y)={(y_{n+1})}_{n\in\NN^*}$. 

Then $\wt{S}$ can be lifted into a morphism $S$ of
$\rL^\infty(E\times\wh F)$, as previously. Furthermore, any function $f$ in
$\rL^\infty(E)$ can be lifted into $f\otimes \indic$ on
$\rL^\infty(E\times\wh F)$, with
$(f\otimes\indic)(x,y)=f(x)$. 

\begin{theorem}\label{T:dilate_markov}
For all $n\in\NN^*$, all $x\in E$ and all $f\in\rL^\infty(E)$ we have
$$
\int_{\wh F}{S}^n(f\otimes\indic)(x,y)\, d\wh\mu(y)=(L^n
f)(x)\,.
$$
\end{theorem}

\begin{proof}
Recall that we noticed in the proof of Theorem
\ref{T:dynamical_markov}, that the mapping $Y$ associated to $\wt T$
played no role in the proof of this theorem, only the mapping $X$ was
of importance. In particular this
implies that Theorem \ref{T:dilate_markov} is true for $n=1$, for the
dynamical systems $\wt T$ and $\wt{S}$ share the same
$X$-mapping. 

By induction, let us assume that the relation
$$
\int_{\wh F}{S}^k(f\otimes \indic)(x,y)\,d\wh\mu(y)=(L^kf)(x)
$$
holds true for all $f\in\rL^\infty(E)$, all $x\in E$ and all $k\leq
n$. Set $\wh{F}_{[2}$ to be the set of sequences ${(y_n)}_{n\geq
2}$ with values in $\wh F$ and $\wh \mu_{[2}$ the restriction of $\wh \mu$ to  $\wh{F}_{[2}$.  We have
\begin{align*}
\int_{\wh F}&{S}^{n+1}(f\otimes \indic)(x,y)\,d\wh\mu(y)=\\
&=\int_{\wh F} {S}^{n}(f\otimes
\indic)\left(X(x,y_1),\Theta(y)\right)\, d\wh \mu(y)\\
&=\int_F\int_{\wh{F}_{[2}}{S}^{n}(f\otimes
\indic)\left(X(x,y_1),y\right)\, d\wh \mu_{[2}(y)\,
d\mu(y_1)\,.
\end{align*}
Put $\wt x=X(x,y_1)$, the above is equal to
\begin{align*}
\ \ \ &\int_F\int_{\wh{F}_{[2}}{S}^{n}(f\otimes
\indic)\left(\wt x,y\right)\, d\wh \mu_{[2}(y)\,
d\mu(y_1)\\
&=\int_F L^n(f)(\wt x)\, d\mu(y_1)\qq\hbox{\ 
(by induction hypothesis)}\\
&=\int_F L^n(f)(X(x,y_1))\, d\mu(y_1)\\
&=L^{n+1}(f)(x)\,.\qq
\end{align*}
\end{proof}

With this theorem and with Theorem \ref{T:dilate_kernel}, we see that every
Markov chain on $E$ can realized as the restriction on $E$ of the
iterations of a deterministic dynamical system $\wt T$ acting on a
larger set. 

\smallskip
The physical interpretation of the construction above is very interesting. It represents a scheme of ``repeated interactions". That is, we know that the result of the deterministic dynamics associated to $\wt T$ on $E\times F$ gives rises to the Markov operator $L$ on $E$. The idea of the construction above is that the environment is now made of a chain of copies of $F$, each of which is going to interact, one after the other, with $E$. After, the first interaction between $E$ and the first copy of $F$ has happened, following the dynamical system $\wt T$, the first copy of $F$ stops interacting with $E$ and is replaced by the second copy of $F$. This copy now interacts with $E$ following $\wt T$. And so on, we repeat these interactions. The space $E$ keeps the memory of the different interactions, while each copy of $F$ arrives independently  in front of $E$ and induces one more step of evolution following $\wt T$. 

As a result of this procedure, successive evolutions restricted to $E$ correspond to iterations of the Markov operator $L$. This gives rise to behavior as claimed: an entire path of the homogeneous Markov chain with generator $L$. 

\subsection{Defect of Determinism and Loss of Invertibility}\label{SS:defect}

We end up this section with some algebraic characterizations of determinism
for Markov chains. The point is to characterize what exactly is lost when going from the deterministic dynamics $T$ on $E\times F$ to the Markov operator $L$ on $E$. 

\begin{theorem}\label{T:deterministic}
Let $(E,\rE)$ be a Lusin space. Let
$\seq X$ be a Markov chain with state space $(E,\rE)$ and with
transition kernel  
$\nu$. Let $L$ be the Markov operator on $\rL^\infty(E)$ associated to $\nu$:
$$
Lf(x)=\int_E f(y)\, \nu(x, dy)\,.
$$
Then the Markov chain $\seq X$ is deterministic if and only if $L$ is a
$*$-homomorphism of the algebra $\rL^\infty(E)$. 
\end{theorem}
\begin{proof}
If the Markov chain is deterministic, then $L$ is associated to a dynamical system and hence it is a $*$-homomorphism (Proposition \ref{P:dynamical}).

\smallskip
Conversely, suppose that $L$ is a $*$-homomorphism. We shall first consider the case where $(E,\rE)$ is a Borel subset of a compact metric space. 

Take any
$A\in\rE$, any $x\in E$ and recall that we always have
$$
\nu(x,A)=L(\indic_A)(x)\,.
$$
The homomorphism property gives
$$
L(\indic_A)(x)=L(\indic_A^2)(x)=L(\indic_A)^2(x)=\nu(x,A)^2\,.
$$
Hence $\nu(x,A)$ satisfies $\nu(x,A)^2=\nu(x,A)$. This means that $\nu(x,A)$
is equal to 0 or 1, for all $x\in E$ and all $A\in \rE$.

\smallskip
Consider a covering of $E$ with a countable family of balls $(B_i)_{i\in\NN}$, each of which with diameter smaller than $2^{-n}$ (this is always possible as $E$ is separable). From this covering one can easily extract a partition $(S_i)_{i\in\NN}$ of $E$ by measurable sets, each of which with diameter smaller than $2^{-n}$. We shall denote by $\rS^n$ this partition. 

Let $x\in E$ be fixed. As we have $\sum_{E\in\rS^n} \nu(x, E)=1$ we
must have $\nu(x,E)=1$ for one and only one $E\in\rS^n$. Let us denote by
$E^{(n)}(x)$ this unique set. Clearly, the sequence
${(E^{(n)}(x))}_{n\in\NN}$ is decreasing (for otherwise there will be
more than one set $E\in\rS^n$ such that $\nu(x, E)=1$). Let $A=\cap_n
E^{(n)}(x)$. The set $A$ satisfies $\nu(x, A)=1$, hence $A$ is non-empty. But also, the diameter of $A$ has to be 0, for it is smaller than $2^{-n}$ for all $n$. As a consequence $A$ has to be a singleton
$\{y(x)\}$, for some $y(x)\in 
E$. Hence we have
proved that for each $x\in E$ there exists a $y(x)\in E$ such that
$n(x, \{y(x)\})=1$. This proves
the deterministic character of our chain.  

\smallskip
The case where $E$ is only homeomorphic to a Borel subset $E'$ of a compact metric space is obtained by using the homeomorphism to transfer suitable partitions $\rS^n$ of $E'$ to $E$. 
\end{proof}

\smallskip
The result above is quite amazing. It gives such a clear and neat characterization of the difference between a true Markov operator and a deterministic one! One can even think of several applications of this characterization, for example one may be able to measure the ``level of randomness" of some Markov operator by evaluating for example
$$
\sup\{\norme{T(f^2)-T(f)^2};\, f\in\rL^\infty (E,\rE), \norme{f}=1\}\,.
$$
I do not know if such things have already been studied or not. It is not my purpose here to develop this idea, I just mention it.
\eject
Another strong result on determinism of Markov chains is the way it is
related to non-invertibility. 

\begin{theorem}\label{T:invertible_Markov}
Let $(E,\rE)$ be a Lusin space. 
Let $L$ be a Markov operator on $\rL^\infty(E)$
associated to a Markov chain $\seq X$. If $L$ is invertible in the category of Markov operators then $\seq X$ is deterministic.
\end{theorem}
\begin{proof}
Recall that a Markov operator $L$ maps positive functions to positive
functions. Hence, in the same way as one proves Cauchy-Schwarz
inequality, we always have
$$
\overline{L(f)}=L(\overline{f})
$$
and
$$
L(\ab{f}^2)\geq L(\bar f)L(f)
$$
(hint: write the positivity of $T(\overline{(f+\l g)}(f+\l g))$ for
all $\l\in\CC$).

\smallskip
Let $M$ be a Markov operator such that $ML=LM=I$. 
We have
$$
\ab{f}^2=\bar f f=M\circ L(\bar f f)\geq M(L(\bar f)L(f))\geq M\circ
L(\bar f)\,M\circ L(f)=\bar f f=\ab{f}^2\,.
$$
Hence we have equalities everywhere above. In particular
$$
M\circ L(\bar f f)=M(L(\bar f)L(f))\,.
$$
Applying $L$ to this equality, gives
$$
L(\bar f f)=L(\bar f)L(f)\,,
$$
for all $f\in\rL^\infty(E)$. 

By polarization it is easy to prove now that $L$ is a homomorphism. By  Theorem
\ref{T:deterministic} it is the Markov
operator associated to a deterministic chain. 
\end{proof}

\smallskip
The result above is more intuitive than the one of Theorem \ref{T:deterministic}, from the point of view of open systems. If the dynamical system $\wt T$ on the large space $E\times F$ is invertible, this invertibility is always lost when projecting on $E$. The fact we do not have access to one component of the coupled system makes that we lose all chance of invertibility. 

\section{Continuous Time}\label{S:SDE_dynamical}

We now leave the discrete-time setup to concentrate on continuous-time dynamical systems. We aim to show that stochastic differential equations are actually a particular kind of continuous-time dynamical systems. In particular they are ``deterministic". The type of dynamical system we shall obtain this way is a continuous-time version of the construction of Theorem \ref{T:dilate_markov}. 

\subsection{Preliminaries}

Let us consider the $d$-dimensional Brownian motion $W$ on its canonical space
$(\O,\rF,\PP)$. This is to say that $\O=C_0(\Rp;\RR^d)$ is the space of continuous functions on $\Rp$ with values in $\RR^d$ and which vanish at 0, equiped with the topology of uniform convergence on compact sets, the $\s$-field $\rF$ is Borel $\s$-field of $\O$ and the measure $\PP$ is the Wiener measure, that is, the law of a $d$-dimensional Brownian motion on $\O$.  The canonical Brownian motion $(W_t)$ is defined by $ W_t(\o)=\o(t)$, for all $\o\in\O$ and all $t\in\Rp$. This is to say, coordinate-wise: $W^i_t(\o)=\o_i(t)$, for $i=1,\ldots, d$. 

\smallskip
We define for all $s\in\Rp$ the \emph{shift} $\theta_s$ 
 as a function from $\O$ to $\O$ by
$$
\theta_s(\o)(t)=\o(t+s)-\o(s)\,.
$$
We define the \emph{shift operator} $\Theta_s$ as follows. If $X$ is any random variable on $\O$ we denote by $\Theta_s(X)$ the random variable $X\circ \theta_s$, \emph{whatever is the state space of }$X$. 
In particular we have $\Theta_s(W_t)=W_{t+s}-W_s\,$.

As the process $Y_t=W_{t+s}-W_s$, $t\in\Rp$, is again a $d$-dimensional Brownian
motion, this implies that the mapping $\theta_s$ preserves the measure
$\PP$. As a consequence $\Theta_s$ is an isometry of $L^2((\O,\rF,\PP);\RR^d)$. 

\begin{lemma}\label{L:predict_theta}
If $H$ is a predictable process in $\RR^d$, then, for all fixed $s\in\Rp$,  the process
$K_t=\Theta_s(H_{t-s})$, $t\geq s$ is also predictable.
\end{lemma}
\begin{proof}
The process $K$ as a mapping from $\O\times[s,+\infty[$ to $\RR^d$ is the
composition of $H$ with the mapping $\phi(\o,t)=(\theta_s(\o),t-s)$
from $\O\times [s,+\infty[$ to $\O\times\Rp$. We just need to check
that $\phi$ is measurable for the predictable $\s$-algebra $\rP$. 

Consider a basic
predictable set $A\times ]u,v]$, with $u<v$ and $A\in\rF_u$,  then 
$$
\phi^{-1}(A\times ]u,v])=\theta_s^{-1}(A)\times ]u+s,v+s]\,.
$$
We just need to check that $\theta_s^{-1}(\rF_u)\subset\rF_{u+s}$. The
$\s$-algebra $\rF_u$ is generated by events of the form
$(W^i_t\in[a,b])$, for $t\leq u$ and $i=1,\ldots, d$. The set $\theta_s^{-1}(W^i_t\in[a,b])$
is equal to $(W^i_{t+s}-W^i_s\in[a,b])$, hence it belongs to $\rF_{u+s}$. 

One needs also to note that
$$
\phi^{-1}(A\times\{0\})=\theta_s^{-1}(A)\times\{s\}\in\rF_s\times\{s\}
$$
for all $A\in\rF_0$. We have proved the predictable character of $K$. 
\end{proof} 

\bigskip
In the following, the norm $\norme{\cdot}_2$ is the $L^2((\O,\rF,\PP);\RR^d)$-norm. For a $\RR^d$-valued predictable process $H$ we put
$$
\int_0^t H_s\cdot dW_s=\sum_{i=1}^d\int_0^t H^i_s\, dW^i_s\,.
$$

\begin{lemma}\label{L:IS_theta}
Let $H$ be a predictable process in $\RR^d$ such that $\int_0^{t+s}\normca{H_u}_2\,
du<\infty$. Then we have
\begin{equation}\label{E:IS_theta}
\Theta_s\left(\int_0^tH_u\cdot
dW_u\right)=\int_s^{t+s}\Theta_s(H_{u-s})\cdot dW_u\,.
\end{equation}
\end{lemma} 
\begin{proof}
If $H$ is an elementary predictable process then the identity
\eqref{E:IS_theta} is obvious from the fact that
$\Theta_s(FG)=\Theta_s(F)\,\Theta_s(G)$ for any scalar-valued $F$ and $G$. A general stochastic
integral $\int_0^t H_s\cdot dW_s $ is obtained as a limit in the norm
$$
\normca{\int_0^tH_s\cdot dW_s}_2=\int_0^t \normca{H_s}_2\, ds\,,
$$
of stochastic integrals of elementary predictable processes.
As $\Theta_s$ is an isometry, it is clear that Equation
\eqref{E:IS_theta} holds true for any stochastic integral.
\end{proof}

\bigskip
Here comes now the main result of this section. Before hands recall the following result on stochastic differential equations (cf \cite{RW2}, Chapter V). Let $f$ be a locally bounded Lipschitz function from $\RR^n$ to $\RR^n$ and $g$ a locally bounded Lipschitz function from $\RR^n$ to $M_{n\times d}(\RR)$. Consider the stochastic differential equation
$$
X_t^x=x+\int_0^tf(X^x_u)\, du+\int_0^tg(X^x_u)\cdot dW_u\,,
$$
which is a shorthand for
$$
\left(X_t^x\right)^i=x^i+\int_0^tf(X^x_u)^i\, du+\sum_{j=1}^d\int_0^tg(X^x_u)^i_j\cdot dW^j_u\,,
$$
for all $i=1,\ldots, n$. Then this equation admits a solution $X^x$ and this solution is unique, in the sense that any other process on $(\O,\rF,\PP)$ satisfying the same equation is almost surely identical to $X^x$.

\begin{theorem}\label{T:X_theta}
Let $W$ be a $d$-dimensional Brownian motion on its canonical space $(\O,\rF,\PP)$.
Let $f$ be a locally bounded Lipschitz function from $\RR^n$ to $\RR^n$ and $g$ a locally bounded Lipschitz function from $\RR^n$ to $M_{n\times d}(\RR)$. Denote by $X^x$
the unique stochastic process (in $\RR^n$) which is a solution of the stochastic differential equation
$$
X_t^x=x+\int_0^tf(X^x_u)\, du+\int_0^tg(X^x_u)\cdot dW_u\,.
$$
Then, for all $s\in\Rp$, for almost all $\o\in\O$, we have, for all $t\in\Rp$,
$$
X^{X^x_s(\o)}_t(\theta_s(\o))=X_{s+t}^x(\o)\,.
$$
\end{theorem}

\noindent{\bf Remark}: Let us be clear about the sentence ``for all $s\in\Rp$, for almost all $\o\in\O$, we have, for all $t\in\Rp$\," above. It means that for all $s\in\Rp$, there exists a null-set $\rN_s\subset \O$ such that for all $w\in\O\setminus\rN_s$ we have, for all $t\in\Rp$ ...

\begin{proof}
Let $s$ be fixed. Define, for all $\o\in\O$
$$
Y^x_u(\o)=\begin{cases}X_u^x(\o)&\mbox{\ if } u\leq
s\,,\\ &\\X^{X^x_s(\o)}_{u-s}(\theta_s(\o))&\mbox{\ if } u>s\,.\end{cases}
$$
Then $Y^x_{s+t}$ satisfies
\begin{align*}
Y^x_{s+t}(\o)&=X^{X_s^x(\o)}_{t}(\theta_s(\o))\\
&=X^x_s(\o)+\left[\int_0^t f(X^{X^x_s}_u)\, du\right](\theta_s(\o))+
\left[\int_0^t g(X^{X^x_s}_u)\cdot dW_u\right](\theta_s(\o))\\
&=x+\left[\int_0^s f(X^x_u)\, du\right](\o)+\left[\int_0^s g(X^x_u)\cdot dW_u\right](\o)+\\
&\ \ \ +\int_0^t
f(X^{X^x_s}_u)(\theta_s(\o))\, du+\left[\int_s^{s+t}\Theta_s\left(g(X^{X^x_s}_{u-s})\right)\cdot dW_u\right](\o)
\end{align*}
by Lemma \ref{L:IS_theta}. 

Now, coming back to the definition of $Y$ we get
\begin{align*}
Y^x_{s+t}(\o)&=x+\int_0^s f(Y^x_u)(\o)\, du+\left[\int_0^s g(Y^x_u)\cdot dW_u\right](\o)+\int_0^t
f(Y^{x}_{u+s})(\o)\, du+\\
&\ \ \ +\left[\int_s^{s+t}g(Y^x_{u})\cdot dW_u\right](\o)\\
&=x+\left[\int_0^s f(Y^x_u)\, du+\int_0^s g(Y^x_u)\cdot dW_u+\int_s^{s+t}
f(Y^{x}_{u})\, du+\right.\\
&\ \ \ +\left.\int_s^{s+t}g(Y^x_{u})\cdot dW_u\right](\o)\\
&=x+\left[\int_0^{s+t} f(Y^x_u)\, du+\int_0^{s+t} g(Y^x_u)\cdot dW_u\right](\o)\,.
\end{align*}
This shows that $Y^x$ is solution of the same stochastic differential
equation as $X^x$. We conclude easily by uniqueness of the
solution.
\end{proof}

\subsection{Stochastic Differential Equations and Dynamical Systems}

We are now ready to establish a parallel between stochastic differential
equations and dynamical systems.  Recall how we defined discrete time
dynamical systems $\wt T$ in Section~\ref{S:markov_dynamical} and their associated semigroups $({\wt T} ^n)$. In continuous
time the definition extends in the following way.

 A
\emph{continuous-time dynamical system} on a measurable space $(E,\rE)$ is a
one-parameter family of measurable functions $(\wt T_t)_{t\in\Rp}$ on $E$ such
that $\wt T_s\circ \wt T_t=\wt T_{s+t}$ for all $s,t$. That is, $\wt T$ is a semigroup
of functions on $E$. 

Each of the mappings $\wt T_t$ can be lifted into an
operator on $\rL^\infty(E)$, denoted by $T_t$ and  defined by 
$$
T_t f(x)=f(\wt T_t\, x)\,.
$$

The following result is now a direct application of Theorem \ref{T:X_theta}.

\begin{corollary}\label{C:SDE_dynamical}
Let $W$ be a $d$-dimensional Brownian motion on its canonical space $(\O,\rF,\PP)$.
Let $f$ be a locally bounded Lipschitz function from $\RR^n$ to $\RR^n$ and let $g$ be a locally bounded Lipschitz function from $\RR^n$ to $M_{n\times d}(\RR)$. Consider the stochastic differential equation (on $\RR^n$)
$$
X_t^x=x+\int_0^tf(X^x_u)\, du+\int_0^tg(X^x_u)\cdot dW_u\,.
$$
Then the mappings $\wt T_t$ on $\RR^n\times \O$ defined by
$$
\wt T_t(x,\o)=\left(X^x_t(\o),\theta_t(\o)\right)
$$
define a continuous time dynamical system on $\RR^n\times\O$, in the sense that there exists a null set $\rN\subset \O$ such that for all $\o\in\O\setminus\rN$, for all $x\in\RR^n$ and for all $s,t\in\Rp$ we have
$$
T_t\circ T_s(x,\o) =T_s\circ T_t(x,\o)=T_{s+t}(x,\o)\,.
$$
\end{corollary}

\begin{proof}
The null set $\rN_s$ appearing in Theorem \ref{T:X_theta} also depends on the initial point $x\in\RR^n$. Let us denote by $\rN_{x,s}$ this set, instead.
Let $\rN_{x,s,t}$ be the null set $\rN_{x,s}\cup\rN_{x,t}$. Finally put 
$$
\rN=\bigcup_{x\in\QQ^n}\bigcup_{s,t\in\QQ^+} \rN_{x,s,t}\,.
$$
Then $\rN$ is a null set and for all $\o\in\O\setminus\rN$ the relations
$$
T_t\circ T_s(x,\o) =T_s\circ T_t(x,\o)=T_{s+t}(x,\o)
$$
hold true for all $x\in\QQ^n$ and all $s,t\in\QQ^+$, by Theorem \ref{T:X_theta}.

The solution $X^x_t(\o)$ is continuous in $t$, except for a null set $\rN'$ of $\o$'s. Hence, by continuity, the relations above remain true for all $s,t\in\Rp$, if $\o\in\O\setminus(\rN\cup\rN')$.  

In the same way, as the solution $X^x_t$ depends continuously in $x$, we conclude easily.
\end{proof}

This is to say that, apart from this minor restriction to the complementary of a null set in $\O$, a stochastic differential equation is nothing more
than a deterministic dynamical system on a product set $\RR^n\times\O$,
that is, it is a semigroup of point transformations of this product
set. 

\medskip We now have a result analogous to the one of Theorem \ref{T:dilate_markov}
when this dynamical system is restricted to the $\RR^n$-component. 
But before establishing this result, we need few technical lemmas.
In the following $\O_{t]}$ denotes the space of continuous functions
from $[0,t]$ to $\RR^d$. For all $\o\in\O$ we denote by $\o_{t]}$ the restriction
of $\o$ to $[0,t]$. Finally $\PP_{t]}$ denotes the restriction of the
measure $\PP$ to $(\O_{t]}\,,\rF_t)$.

\begin{lemma}\label{L:o_thetao}
The image of the measure $\PP$ under the mapping 
$$
\begin{matrix} \O&\rightarrow&\O_{t]}\times\O\\
\o&\mapsto&(\o_{t]}\,,\theta_t(\o))
\end{matrix}
$$
is the measure $\PP_{t]}\otimes \PP$.
\end{lemma}
\begin{proof}
Recall that $\o(s)=W_s(\o)$ and
$\theta_t(\o)(s)=W_{t+s}(\o)-W_t(\o)$. If $A$ is a finite
cylinder of $\O_{t]}$ and $B$ a finite cylinder of $\O$, then the set
$$
\{\o\in\O\,;\ (\o_{t]}\,,\theta_t(\o))\in A\times B\}
$$
is of the form
\begin{multline*}
\{\o\in\O\,;\ W_{t_1}(\o)\in A_1, \ldots, W_{t_n}(\o)\in A_n,(W_{s_1}-W_t)(\o)\in
B_1, \ldots\hfill\\
\hfill\ldots, (W_{s_k}-W_t)(\o)\in B_k\}
\end{multline*}
for some $t_1,\ldots, t_n\leq t$ and some $s_1,\ldots, s_k> t$. By the
independence of the Brownian motion increments, the probability of the
above event is equal to 
\begin{multline*}
\PP_{t]}(\{\o\in\O_{t]}\,;\ W_{t_1}(\o)\in A_1, \ldots, W_{t_n}(\o)\in
A_n\})\times\hfill\\
\hfill \times\PP(\{\o\in\O\,;\ (W_{s_1}-W_t)(\o)\in
B_1, \ldots, (W_{s_k}-W_t)(\o)\in B_k\})\,.
\end{multline*}
This is to say,
\begin{multline*}
\PP(\{\o\in\O\,;\ (\o_{t]}\,,\theta_t(\o))\in A\times B\})=\hfill\\
\hfill=\PP_{t]}(\{\o_{t]}\in\O_{t]}\,;\
\o_{t]}\in A\})\,\PP(\{\o\in\O\,;\ \theta_t(\o)\in B\})\,.
\end{multline*}
This is exactly the claim of the lemma for the cylinder sets. As the
measures $\PP$ and $\PP_{t]}\otimes \PP$ are determined by their values on the
cylinder sets, we conclude easily. 
\end{proof}

\begin{lemma}\label{L:g(o,oprime)}
Let $g$ be a bounded measurable function on $\O_{t]}\times\O$. Then we
have
$$
\int_\O\int_{\O_{t]}} g(\o, \o')\, dP_{t]}(\o)\, dP(\o')=\int_\O
g(\o_{t]}, \theta_t(\o))\, dP(\o)\,.
$$
\end{lemma}
\begin{proof}
This is just the Transfer Theorem for the mapping of Lemma \ref{L:o_thetao}.
\end{proof}

\begin{theorem}\label{T:SDE_dilate}
Let $(\O,\rF,\PP)$ be the canonical space of a $d$-dimensional Brownian motion
$W$. Let $f$ be a locally bounded Lipschitz function from $\RR^n$ to $\RR^n$ and let $g$ be a locally bounded Lipschitz function from $\RR^n$ to $M_{n\times d}(\RR)$. Consider the stochastic differential equation (on $\RR^n$)
$$
X_t^x=x+\int_0^t f(X^x_u)\, du+\int_0^tg(X^x_u)\cdot dW_u
$$
and the associated dynamical system 
$$
\wt T_t(x,\o)=\left(X^x_t(\o),\theta_t(\o)\right)\,.
$$
For any bounded function $h$ on $\RR^n$ consider the mapping
$$
P_t\,h(x)=\EE\left[\,T_t(h\otimes \indic)(x,\,\cdot\,)\right]=\int_\O h(X^x_t(\o))\, d\PP(\o)\,.
$$
Then $(P_t)_{t\in\Rp}$ is a Markov semigroup on $\RR^n$ with generator
$$
A=\sum_{i=1}^n f_i(x)\,\frac \partial{\partial x_i}+\frac 12\sum_{i,j=1}^n \sum_{\a=1}^d g^i_\a (x) g^j_\a(x)\,\frac{\partial^2}{\partial x_i\,\partial x_j}\,.
$$
\end{theorem}
\begin{proof}
The fact that each $P_t$ is a Markov operator is a consequence of
Theorem \ref{T:dynamical_markov}. Let us check that they form a
semigroup. 

First of all note that, since $X^x$ is a predictable process, the quantity $X^x_t(\o)$ depends only on $\o_{t]}$ and not on the whole of $\o$. We shall denote by $X^x_t(\o_{t]})$ the associated function of $\o_{t]}$. 

By definition of $P_t$ we have
\begin{align*}
P_t(P_s\,h)(x)&=\EE\left[T_t(P_sh\otimes
\indic)(x,\,\cdot\,)\right]\\
&=\int_\O P_sh(X^x_t(\o))\,d\PP(\o)\\
&=\int_\O\,\int_\O h\left(X^{X^x_t(\o)}_s(\o')\right)\, d\PP(\o')\, d\PP(\o)\\
&=\int_{\O_{t]}}\,\int_\O h\left(X^{X^x_t(\o)}_s(\o')\right)\, d\PP(\o')\, d\PP_{t]}(\o)\\
&=\int_\O h\left(X^{X^x_t(\o_{t]})}_s(\theta_t(\o))\right)\, d\PP(\o)\qq\mbox{(by
Lemma \ref{L:g(o,oprime)})}\\
&=\int_\O h\left(X^{x}_{s+t}(\o)\right)\, d\PP(\o)\qq\mbox{(by
Theorem \ref{T:X_theta})}\\
&=P_{s+t}h(x)\,.
\end{align*}
We have proved the semigroup property. 

The rest of the proof comes from the usual theory of Markov semigroups and their associated generators (see for example \cite{R-Y}, Chapter VII). 
\end{proof}

\bigskip
We have proved the continuous time analog of Theorem
\ref{T:dilate_markov}. Every Markov semigroup, with a generator of the
form above, can be dilated on a larger set (a product set) into a deterministic 
dynamical system. What is maybe more surprising is that the deterministic dynamical system in question is a stochastic differential equation. Theorem \ref{T:SDE_dilate}  and Corollary \ref{C:SDE_dynamical} show that a stochastic differential equation can actually be seen as a particular deterministic dynamical system. 

Theorem \ref{T:SDE_dilate} above again gives an open system point of view on
Markov processes: Markov processes are obtained by the restriction of
certain types of dynamical systems on a product space, when one is
averaging over one inaccessible component. The role of the environment is now played by the Wiener space and the role of the global dynamics on the product space is played by the stochastic differential equation.

\smallskip
In this section we have developed the Brownian case only. But it is
clear that all this discussion extends exactly in the same way to the
case of the Poisson process. Indeed, the arguments developed above are mostly 
only based on the independent increment property. 

\smallskip
We have said that stochastic differential equations are particular dynamical systems which are continuous analogues of those of Section \ref{SS:RI}: repeated interactions. In the article \cite{julien}, the convergence of discrete-time repeated interactions models to stochastic differential equations is proved. 

\par\bigskip\noindent
{\bf Comment.} We do not pretend that all the results presented in this article are new. Let us be clear about that. The fact that restrictions of dynamical systems can give rise to Markov chains is rather well-known among specialists of dynamical systems. The results of Subsection \ref{SS:defect} are adaptations to the classical context of similar results on completely positive maps for quantum systems. The fact that stochastic differential equations give rise to deterministic dynamical systems is also not new and can be found for example in \cite{Car} (see also \cite{App} for more general noises). 

The originality of our article lies more in its survey character, in the way we put all these results together, in the connection we make with repeated interaction systems and in the physical point of view we adopt. 

\par\bigskip\noindent
{\bf Acknowledgment.} The author is very grateful to the referee of this article for his very carefull reading, his remarks and suggestions.

\bibliographystyle{amsplain}

\begin{thebibliography}{99}


\bibitem{App}
D. Applebaum: \emph{L\'evy processes and stochastic calculus}, Second edition, Cambridge Studies in Advanced Mathematics, 116. Cambridge University Press, Cambridge, 2009.

\bibitem{A-P}
S. Attal, Y. Pautrat: \emph{From Repeated to Continuous Quantum Interactions}, Annales Henri Poincar\'e (Theoretical Physics) 7 (2006), pp.\thinspace 59-104.

\bibitem{O-B}
C. Bernardin, S. Olla: \emph{Fourier's law for a microscopic heat conduction model}, Journal of Statistical Physics 121 (2005), pp.\thinspace 271-289.
 
\bibitem{Bod}
T. Bodineau, G. Giacomin: \emph{From dynamic to static large deviations in boundary driven exclusion particle systems},  Stochastic Processes and Applications 110 (2004), pp.\thinspace 67-81.

\bibitem{Rey}
F. Bonetto, J. L. Lebowitz, L. Rey-Bellet: \emph{Fourier Law: A challenge to theorists}, In: Mathematical Physics 2000, A. Fokas, A. Grigoryan, T. Kibble, and B. Zegarlinski (Eds.) Imp. Coll. Press, London 2000.

\bibitem{Car}
A. Carverhill: \emph{Flows of stochastic dynamical systems: ergodic theory} Stochastics 14 (1985), pp.\thinspace 273Ð317. 

\bibitem{julien}
J. Deschamps, \emph{Continuous Time Limit of Classical Repeated Interaction Systems}, preprint.

\bibitem{Spe}
F. Haake, D. Spehner: \emph{Quantum measurements without macroscopic superpositions}, 
Physical Review A 77, 052114 (2008), 24 pp.

\bibitem{JOP}
V. Jaksic, Y. Ogata, C.-A. Pillet: \emph{The Green-Kubo formula and the Onsager reciprocity relations in quantum statistical mechanics},
Communications in Mathematical Physics 265, 3 (2006), pp.\thinspace 721-738.

\bibitem{K-P}
D. Karevski, T. Platini: \emph{Quantum Non-Equilibrium Steady States Induced by Repeated Interactions}, Phys Rev. Letter, to appear.

\bibitem{Pre}
J. Preskill: \emph{Quantum Information and Quantum Computation}, Course Files on the Web: http://www.theory.caltech.edu/people/preskill/ph229

\bibitem{R-Y}
D. Revuz, M. Yor: \emph{Continuous Martingales and Brownian Motion}, Springer Verlag, Grundlehren der mathematischen Wissenschaften 293 (2005).

\bibitem{RW2}
L.C.G. Rogers, D. Williams: \emph{Diffusions, Markov Processes and Martingales, Volume 2}, Cambridge University Press (2000).

\bibitem{Rue}
D. Ruelle: \emph{A departure from equilibrium}, Nature 414 no. 6861 (2001), pp.\thinspace 263-264.




\end{thebibliography}
\vfill\eject

\end{document}